\documentclass[11pt]{amsart}
\usepackage{graphicx} 
\usepackage[margin=3cm]{geometry}
\usepackage{mathpazo}
\usepackage{amssymb}
\usepackage{todonotes}
\usepackage{orcidlink}
\usepackage[normalem]{ulem}
\usepackage{cancel}

\newtheorem{theorem}{Theorem}[section]
\newtheorem{proposition}[theorem]{Proposition}
\newtheorem{lemma}[theorem]{Lemma}
\newtheorem{corollary}[theorem]{Corollary}

\newtheorem{conjecture}[theorem]{Conjecture}

\theoremstyle{remark}

\newtheorem{example}[theorem]{Example}

\newcommand{\eval}{\text{eval}}
\newcommand{\betti}{\operatorname{Betti}}

\begin{document}

\title[Full atomic support]{Betti elements and full atomic support in rings and monoids}

	\author[S.~T.~Chapman]{Scott T. Chapman \,\orcidlink{0000-0003-1205-6899}}
\address{Department of Mathematics and Statistics\\Sam Houston State University\\Huntsville, TX 77341}
\email{scott.chapman@shsu.edu}
	
	\author[P.~Garc\'a-S\'anchez]{Pedro Garc\'ia-S\'anchez \,\orcidlink{0000-0003-2330-9871}}
	\address{Departamento de Álgebra, Universidad de Granada, Granada, Spain}
	\email{pedro@ugr.es}

	\author[C.~O'Neill]{Christopher O'Neill \,\orcidlink{0000-0001-7505-8184}}
	\address{Department of Mathematics and Statistics, San Diego State University, San Diego, California 92182-7720}
	\email{cdoneill@sdsu.edu}

	\author[V.~Ponomarenko]{Vadim Ponomarenko \,\orcidlink{0000-0002-1174-3366}}
	\address{Department of Mathematics and Statistics, San Diego State University, San Diego, California 92182-7720}
	\email{vadim123@gmail.com}

\date{\today}

\begin{abstract}
Several papers in the recent literature have studied factorization properties of affine monoids using the monoid's Betti elements.  In this paper, we extend this study to more general rings and monoids.  We open by demonstrating the issues with computing the complete set of Betti elements of a general commutative cancellative monoid, and as an example compute this set for an algebraic number ring of class number two.  We specialize our study to the case where the monoid has a single Betti element, before examining monoids with full atomic support (that is, when each Betti element is divisible by every atom).  For such a monoid, we show that the catenary degree, tame degree, and omega value agree and can be computed using the monoid's set of Betti elements.  We close by considering Betti elements in block monoids, giving a "Carlitz-like" characterization of block monoids with full atomic support and proving that these are precisely the block monoids having a unique Betti element. 	 
\end{abstract}


\maketitle

\section{Introduction}

The study of various properties related to nonunique factorizations of elements in commutative rings and monoids has been an active area of research in algebra, combinatorics, and number theory over the past 40 years.  One specific area of interest has been the computation of various combinatorial constants which help describe the nonunique arithmetic of many classes of integral domains and monoids.  In \cite{c-t}, the authors show that several of these constants can be computed (and in some cases are even equal) on a monoid $S$ by only considering the Betti elements of $S$.  For many affine monoids (in particular, numerical monoids) this made the computation of constants such as the catenary and tame degrees into a finite time problem.  Hence, over the last decade several papers (see for example 
\cite{omega,delta-bf,measure-w,CJMM,isolated,half-factorial,single,ph}) consider issues where Betti elements play a key role.  Most of these papers focus on the study of Betti elements in affine or finitely generated commutative monoids.  In this paper, we generalize this study and consider Betti elements in various types of rings and monoids.  We investigate their structure and properties, as well as  their use for determining factorizations properties of their base structures.  In particular, we consider monoids $S$ where each Betti element is divisible by each atom of $S$ and denote these monoids as having full atomic support.  We show the relationship between full atomic support and two properties explored in the recent literature:\ the single Betti element case, and the length factorial case, in the diagram below.  We also offer examples that illustrate that these implications cannot be reversed.  

\begin{equation}\label{mainresult}
\mbox{\begin{tabular}{ccccccc}
$S$ is length factorial 
& $\Longrightarrow$ & 
\begin{tabular}{c}
$S$ has at most \\ one Betti element
\end{tabular}
& $\Longrightarrow$ & 
\begin{tabular}{c}
$S$ has full atomic support.
\end{tabular} 
\end{tabular}}
\end{equation}

We begin in Section~\ref{sec:basic} by offering the necessary background and definitions.  We argue (in Proposition~\ref{betti-half-factorial}) that a monoid $S$ is half-factorial if and only if each Betti element of $S$ is half-factorial.  In Section~\ref{sec:examples}, we offer many examples of the computation of the complete set of Betti elements in various types of monoids.  We offer particular focus on block monoids and algebraic number rings of class number two.  In Section~\ref{sec:one-betti}, we focus on monoids that admit only one Betti element.  
We argue (in Lemma~\ref{cor:atom-divides-single-betti}) that the Betti element in such a monoid determines a partition of the set of atoms of $S$.  In Theorem~\ref{char1BE}, we characterize monoids with one Betti element $b$ via the multiplicity of the atoms that appear in various irreducible factorizations of $b$ (we define this multiplicity as the \textit{multiplicative shadow} of $S$).  
We show in Theorem~\ref{anymultshadow} that any possible multiset of positive integers appears as the multiplicative shadow of a monoid with one Betti element.  Section~\ref{sec:one-betti} closes (in Theorem~\ref{lfobe}) by showing that a length factorial monoid is merely a monoid with a single Betti element that has exactly two irreducible factorizations which have differing length.  

In Section~\ref{sec:full-as}, we explore the notion of \textit{full atomic support}.  An element $x\in S$ has full atomic support if it is divisible by each atom of $S$; we say that $S$ has \textit{full atomic support} if each Betti element of $S$ has full atomic support.  
In Proposition~\ref{prop:fr-bm}, we show that the Betti elements in a monoid with full atomic support are incomparable under the divisibility order.  This eventually leads to an argument (Lemma~\ref{lem:c-fr}, Lemma~\ref{lem:fr-ZaS}, and Theorem~\ref{th:c=t=w-full-rank}) showing that in a monoid $S$ with full atomic support, the catenary degree, the tame degree, and the omega values are all equal and can be computed using the set of Betti elements of $S$.  Section~\ref{sec:block} looks more closely at the set of Betti elements of certain block monoids.  Using some technical constructions grounded in Lemma~\ref{full-rank-finite}, Lemma~\ref{lem:blockmonoid}, and Theorem~\ref{thm:blockmonoid}, we offer in Corollary~\ref{cor:blockmonoid} a ``Carlitz-like'' characterization of block monoids with full atomic support.


\section{Basic facts and definitions}
\label{sec:basic}

Let $\mathbb{N}$ represent the positive integers, $\mathbb{N}_0$ the nonnegative integers, and $(S,+)$ a commutative monoid.  We say $S$ is \emph{cancellative} if $a+b=a+c$ implies $b=c$ for all $a, b, c, \in S$, and that $S$ is \emph{reduced} if the only unit in $S$ is its identity element.  We say that $a \in S$ is an \textit{atom} if $a = b+c$ for some $b,c \in S$ implies $b$ or $c$ is a unit, and we say that $S$ is \textit{atomic} if every nonzero nonunit element of $S$ can be written as a sum of atoms.  
Suppose $S$ is cancellative and reduced.  Given $a,b\in S$, we write $a\le_S b$ and say ``$a$ divides $b$ in $S$'' if there exists $c\in S$ such that $a+c=b$ (equivalently, $b\in a+S$). The binary relation $\le_S$, called the \emph{divisibility order} on $S$, is a partial order on $S$ (reflexive since $0 \in S$, transitive since $S$ is a monoid, and anti-symmetric since $S$ is cancellative and reduced). We will write $a<_S b$ when $a\le_S b$ and $a\neq b$.  

In the study of non-unique factorization, it is common to restrict to the case of reduced monoids (see the discussion after \cite[Definition~1.2.6]{GHKb}). 
Hence, throughout the remainder of this manuscript, unless otherwise stated, all monoids are assumed to be atomic, cancellative, commutative, and reduced.  Additionally, we restrict our attention to monoids satisfying the ascending chain condition on principal ideals (ACCP), that is, its divisibility order does not contain infinite descending chains (we will see in Example~\ref{e:nonaccpi} the motivation for this assumption).

If $\mathcal{A}(S)$ is the set of atoms of $S$, 
then we know that $\langle \mathcal{A}(S)\rangle = S$. 
We say that $a\in \mathcal A(S)$ is \textit{prime} if whenever $a\le_S b+c$ for some $b,c\in S$, either $a\le_S b$ or $a\le_S c$. 
Let $\mathcal{F}(\mathcal{A}(S))$ be the free (commutative) monoid over $\mathcal{A}(S)$, which we can identify with $\mathbb{N}_0^{(\mathcal{A}(S))}$, and
\begin{equation}\label{eq:fact-hom}
\varphi: \mathcal{F}(\mathcal{A}(S)) \to S, a \in \mathcal{A} \mapsto a,
\end{equation}
be the factorization morphism of $S$. 
For $s\in S$, denote by $\mathsf{Z}(s)$ the set of factorizations of $s$, that is, $\mathsf{Z}(s)=\varphi^{-1}(s)$.  If $z\in\mathsf{Z}(s)$ with
$z=n_1a_1 + \cdots + n_ka_k$ for distinct atoms $a_i$ and natural numbers $n_j$, then we call $|z|= n_1+\cdots+n_k$ the \textit{length} of $z$.
 
Let $P$ be the set of prime elements of $S$, and let $T=S\setminus (\bigcup_{a\in P} a+S)$. By \cite[Theorem~1.2.3]{GHKb}, $T$ is a submonoid of $S$ and $S$ is isomorphic to $\mathcal{F}(P)\times T$. 
If $N=\mathcal{A}(S)\setminus P$, then $N\subseteq T$. If~$s\in T$, then $s=\sum_{a\in \mathcal{A}(S)} s_a a$, and as $s\not \in \bigcup_{a\in P} a+S$, we deduce that $s_a=0$ for all $a\in P$. Thus, $s\in \langle N\rangle$. This proves that $T=\langle N\rangle$, that is, 
\[S\simeq \mathcal{F}(P)\times \langle \mathcal{A}(S)\setminus P\rangle.\]
Thus, the study non-unique factorizations focuses on $\langle \mathcal{A}(S)\setminus P\rangle$ and, without loss, we can normally assume that our monoid $S$ does not have prime elements.

The kernel $\ker(\varphi)$ is an equivalence relation $\sim$ on $\mathcal{F}(\mathcal{A}(S))$ with $z \sim z'$ whenever $\varphi(z) = \varphi(z')$.  That is, whenever $z$ and $z'$ are factorizations of the same element of $S$.  In fact, $\ker(\varphi)$ is a \emph{congruence}, meaning $z \sim z'$ implies $z + z'' \sim z' + z''$ for all $z, z', z'' \in \mathcal{F}(\mathcal{A}(S))$.  
This ensures the set $\mathcal{F}(\mathcal{A}(S))/\ker(\varphi)$ of equivalence classes of $\ker(\varphi)$ inherits a monoid operation from $\mathcal{F}(\mathcal{A}(S))$, and is naturally isomorphic to $S$.  
A \emph{presentation} of $S$ is a system of generators of $\ker(\varphi)$ (as a congruence). A presentation $\sigma$ is \emph{minimal} if no proper subset of $\sigma$ generates $\ker(\varphi)$.

We define the \emph{atomic support} of $s \in S$, written $\operatorname{AS}(s)=\{ a\in \mathcal{A}(S) : s\in a+S\}$, as the set of atoms which divide $s$. If $x$ is an atom, then $\operatorname{AS}(x)=\{x\}$. For $z\in \mathsf{Z}(s)$, we similarly define its \emph{support} as $\operatorname{supp}(z)=\{a\in \mathcal{A}(S):z_a\neq 0\}$.  Note that $\operatorname{AS}(s)=\bigcup_{z\in\mathsf{Z}(s)}\operatorname{supp}(z)$. Unlike $\operatorname{AS}(s)$, observe that $\operatorname{supp}(y+z)=\operatorname{supp}(y)\cup\operatorname{supp}(z)$, and in particular that $\operatorname{supp}(z)=\operatorname{supp}(k z)$ for every positive integer $k$. 

An element $s\in S$ is said to be a \emph{Betti element} if there exists a partition $\operatorname{AS}(s)=A_1\cup \dots \cup A_n$, with $n$ an integer greater than one, such that for any factorization $z$ of $s$, there exists $i\in \{1,\dots,n\}$ so that $\operatorname{supp}(s)\subseteq A_i$. The set of the Betti elements of $S$ is denoted by $\betti(S)$.   

Betti elements can also be viewed using graph theory.  For $s\in S$, let $\nabla_s$ be the graph with vertices $\mathsf{Z}(s)$, and whose edges are the pairs $(x, y)$ such that $x\cdot y \neq 0$ (that is, edges join factorizations having atoms in common; here $x\cdot y$ denotes the dot product of $x$ and $y$). The connected components of $\nabla_s$ are called the $\mathcal{R}$-classes of $s$. Observe that the Betti elements of $S$ are precisely those elements with at least two $\mathcal{R}$-classes.  
Notice that if $S$ has finitely many atoms ($S$ is \emph{finitely generated}), then $S$ admits a finite presentation by R\'edei's Theorem \cite{redei}, and thus $S$ has finitely many Betti elements.

Betti elements and their $\mathcal{R}$-classes can be used to characterize the minimal presentations of~$S$, and, thus, they are widely studied in the theory of commutative monoids (see for instance \cite{acpi, delta-bf, single, Her}). 
According to \cite{acpi}, if $S$ satisfies the ACCP, a minimal presentation of $S$ can be constructed from its Betti elements as follows. 
Let $(X_i,Y_i)$, $i \in I$ be the edges of a tree whose vertices are the connected components of $\nabla_b$.
Pick $x_i \in X_i$ and $y_i \in Y_i$ for all $i \in I$ and set $\rho^{(b)} = \{(x_i , y_i ) : i \in I \}$. Then, $\rho = \bigcup_{b\in\betti(S)} \rho^{(b)}$ is a minimal
presentation of $S$ (see \cite{isolated} for more details) and 
all minimal presentations have this form. 
Note that the ACCP is critical for this process to yield a presentation; see Example~\ref{e:nonaccpi} for a demonstration of what can occur in its absence.  

We say a non-unit $s\in S$ in an atomic monoid $S$ is \emph{half-factorial} if all factorizations of $s$ have equal length, and that $S$ is \emph{half-factorial} if every $s\in S$ is half-factorial \cite[Section~6.7]{GHKb}.
The following result can be seen as a non-finitely generated version of \cite[Proposition~22]{elasticity-atomic} and \cite[Lemma~1.2]{half-factorial}. For monoids in which every element has a finite set of distinct factorization lengths, the statement follows from \cite[Theorem~2.5]{delta-bf}.

\begin{proposition}\label{betti-half-factorial} 
A monoid $S$ is half-factorial if and only if every Betti element of $S$ is half-factorial.    
\end{proposition}
\begin{proof} 
One direction is trivial.  For the other, assume that every element of $\betti(S)$ is half-factorial.  Let $s\neq 0$ be in $S$ and let $u,v\in \mathsf{Z}(s)$. 
By \cite[Theorem~1]{acpi} there is some chain of factorizations $u=z_1, z_2, \ldots, z_t=v$, where $(z_i,z_{i+1})=(a_i+c_i,b_i+c_i)$ and $\varphi(a_i)=\varphi(b_i)\in \betti(S)$, where $\varphi$ is defined as in \eqref{eq:fact-hom}. For each $i$, we have $|z_i|=|a_i|+|c_i|=|b_i|+|c_i|=|z_{i+1}|$, since $\phi(a_i)$ is half-factorial. Hence $|u|=|v|$.
\end{proof}


\section{Motivating examples}
\label{sec:examples}

In this section, we present several families of monoids that will be used throughout the manuscript.  

\begin{example} Suppose $n_1<n_2< \dots < n_k$ are positive integers, with $\gcd(n_1,\dots,n_k)=1$.
The monoid 
\[
S
= \langle n_1, n_2, \dots , n_k\rangle
= \{z_1n_1 + \cdots + z_kn_k : z_1, \ldots, z_k \in \mathbb N_0\}
\subseteq (\mathbb N_0, +)
\]
is called a \emph{numerical monoid} whose atoms are $n_1, n_2, \dots, n_k$. The \emph{embedding dimension} of $S$ is $\operatorname{e}(S)=k \geq 2$ and its \emph{multiplicity} is $\operatorname{m}(S)=n_1$.  As described above, the computation of the Betti elements of $S$ is closely related to the computation of a minimal presentation of $S$. By \cite[Corollary 8.27]{ns}, the size of such a minimal presentation is at most
\begin{equation}\label{upperbnd}
\tfrac{1}{2}\operatorname{m}(S)(\operatorname{m}(S)-1)
\end{equation}
and hence, this figure is also an upper bound on the number of Betti elements of $S$.  If $S$ is of maximal embedding dimension (that is, if $\operatorname{e}(S) = \operatorname{m}(S)$), then the upper bound \eqref{upperbnd} for the size of the minimal presentation is attained (see \cite[Lemma 8.29]{ns}).  Thus, for such an $S$, any Betti element $b$ is the sum of two atoms of $S$ (see \cite[Proposition 8.19]{ns}).  For instance, if $t\geq 2$, then 
\[
S = \langle t, t+1, \ldots , 2t-1\rangle
\qquad \text{has} \qquad
\betti(S)=\{2t, 2t + 1, \ldots , 4t-2\}.
\]

We note the special case when $k=3$ above (that is, $S=\langle n_1, n_2, n_3\rangle$). By a fundamental result of Herzog \cite{Her}, the cardinality of a minimal presentation in this case is at most three and hence, such an $S$ has at most three Betti elements. We look at conditions which determine each possibility.
\begin{enumerate} 
\item \textbf{One Betti Element:} Here $S=\langle k_1k_2, k_1k_3, k_2k_3\rangle$ and $\betti(S)= \{k_1k_2k_3\}$, where $k_1, k_2, k_3$ are pairwise relatively prime positive integers (see \cite[Example 12]{single}).
\item \textbf{Two Betti Elements:} In this case, $S = \langle am_1, am_2, bm_1 + cm_2\rangle$ where $1< m_1 < m_2$ are relatively prime integers, 
$a, b$, and $c$ are nonnegative integers with $a\geq 2$, $b + c \geq 2$
and $\gcd(a, bm_1 + cm_2) = 1$.  Here $S$ is symmetric and $\betti(S)=\{a(bm_1 + cm_2), am_1m_2\}$ (see \cite[Theorem 10.6]{ns} and \cite[Theorem 6]{GSMC}).
\item \textbf{Three Betti Elements:} Here $S$ is nonsymmetric and the generators $n_1, n_2$, and $n_3$ entail all possibilities not covered in the first two cases.  
\end{enumerate}
\end{example}

\begin{example}  As demonstrated above, the question of finding the set
$\betti(S)$ for a monoid $S$ is largely a computational problem.  There are recent cases in the literature where this set is precisely described for particular monoids.  A \textit{positive monoid} $S$ is a submonoid of the additive monoid of nonnegative real numbers. A positive monoid consisting entirely of rationals is known as a \textit{Puiseux monoid}. 
Let $q$ be a non-integer positive rational such that $q^{-1}\not\in\mathbb{N}$, and let $\mathsf{n}(q)$ and $\mathsf{d}(q)$ represent respectively the numerator and denominator of $q$ when written in lowest terms.
Let $S_q := \langle q^n : n \in \mathbb{N}_0 \rangle$ be the Puiseux monoid generated by the powers of $q$ (see \cite[Example 4.6]{CJMM}).  By \cite[Theorem~6.2]{GG18} and 
\cite[Theorem~4.2]{CG22}, $S_q$ is atomic and $\mathcal{A}(S_q) = \{q^n : n \in \mathbb{N}_0\}$. It follows from \cite[Lemma~4.3]{ABLST} that 
\[
\text{Betti}(S_q) = \left\{ \mathsf{n}(q) q^n : n \in \mathbb{N}_0 \right\}.
\]

Let $\mathbb{N}_0[X]$ represent the semiring of polynomials with nonnegative integer coefficients.  If $\alpha$ is a positive real number, then let $\mathbb{N}_0[\alpha] =\{p(\alpha) :  p(X)\in \mathbb{N}_0[X]\}$. It follows that $\mathbb{N}_0[\alpha]$ is a positive monoid generated over $\mathbb{N}_0$ by the powers of $\alpha$.  If $q$ is a positive rational that is neither an integer nor the reciprocal of an integer, and $\alpha=\sqrt[n]{q}$ a positive irreducible $n$th root of $q$, then Proposition 4.4 in \cite{ABLST} shows that 
\[
\betti(\mathbb{N}_0[\alpha]) = \left\{ \frac{\mathsf{n}(q)^{m+1}}{\mathsf{d}(q)^{m}}\cdot \alpha^r :  m\in\mathbb{N}_0\mbox{  and  } r\in\{0,1,\ldots ,n-1\} \right\}.
\]
\end{example}

\begin{example}\label{lfex}  Let $G$ be an abelian group and let $\mathcal{F}(G)$ denote the free abelian monoid on $G$.
Define $\eval : \mathcal{F}(G) \to G$ by
\[
\eval\left(\prod_{g \in G} g^{\alpha_g}\right) = \sum_{g \in G} \alpha_gg,
\]
where the addition takes place in $G$. The set
\[
\mathcal{B}(G) = \left\{x \in \mathcal{F}(G) : \eval(x) = 0 \right\}
\]
is a submonoid of $\mathcal{F}(G)$ called the {\it block monoid} on the abelian group $G$.  Block monoids play a key role in the theory of non-unique factorizations (see \cite[Sections 2.5, 3.4]{GHKb}).  If $G$ is finite, then $\mathcal{B}(G)$ is an affine monoid and for relatively small $G$, the computation of their Betti Elements can be handled using the GAP \cite{gap} package numericalsgps \cite{numericalsgps}.  We note that both $\mathcal{B}(\{e\})$ and $\mathcal{B}(\mathbb{Z}_2)$ are factorial monoids and admit no Betti elements.  We list the results for $\mathcal{B}(\mathbb{Z}_3)$, $\mathcal{B}(\mathbb{Z}_2\oplus \mathbb{Z}_2)$, and $\mathcal{B}(\mathbb{Z}_4)$ in the following table. 
\medskip

\[
\begin{array}{c|c|c|c}
G & \text{Defining equations of } \mathcal{B}(G)  & \text{Minimal Generators for} & \text{Betti Elements} \\
& \text{as an affine monoid in } \mathbb{N}_0^k & \mathcal{B}(G) & \text{of } \mathcal{B}(G) \\[5pt] \hline  \rule{0pt}{1.2\normalbaselineskip}
\mathbb{Z}_3 & x_1+2x_2\equiv 0 \pmod{3} & (0,3), (3,0), (1,1) & (3,3) \\[5pt] 
\rule{0pt}{1.2\normalbaselineskip}
\mathbb{Z}_2\oplus \mathbb{Z}_2 & x_1+ x_3\equiv 0\pmod{2}, 
& (2,0,0),(0,2,0), & 
(2,2,2)\\
 &  x_2+x_3\equiv 0\pmod{2} & (0,0,2),(1,1,1) & \\[5pt] 
 \rule{0pt}{1.2\normalbaselineskip}
\mathbb{Z}_4 & x_1+2x_2+3x_3\equiv 0 \pmod{4}
 & (0, 0, 4 ), ( 0, 1, 2 ),  & 
(0, 2, 4), (2, 1, 4),  \\
& & ( 1, 0, 1 ), ( 2, 1, 0 ), & (4, 0, 4), (4, 1, 2), \\
& & (0,2,0),(4,0,0) &  (2,2,2), (4,2,0) 
\end{array}
\]
\end{example}

As one might expect, the number of Betti elements of $\mathcal{B}(G)$ increases rapidly with the size of $G$ (GAP shows that $\mathcal{B}(\mathbb{Z}_2^3)$ admits 99 Betti elements). In fact, examples obtained using GAP lead to the following conjecture. 
    
\begin{conjecture}\label{vadimconj} Let $G$ and $G'$ be finite abelian groups of order greater than 2.  Apart from $\mathbb{Z}_3$ and $\mathbb{Z}_2\oplus \mathbb{Z}_2$ above, if $G\not\cong G'$, then the number of Betti elements of $\mathcal{B}(G)$ and $\mathcal{B}(G')$ are different.
\end{conjecture}

Given $s \in S$, we say that a factorization $x \in \mathsf{Z}(s)$ is \emph{isolated} if it is disjoint with every other factorization
of $s$, that is, the $\mathcal{R}$-class of $x$ in $\nabla_s$ is a singleton. It easily follows that if $s$ is an element of $S$ having an isolated factorization, then either $\mathsf{Z}(s)$ is a singleton, or $s$ is a Betti element. Isolated factorizations will play a central role in many of the key results throughout the remainder of this manuscript.

\begin{example}  Let $D$ be the ring of integers in a finite extension of the rationals with class number two. The ring $D$ is a well-known half-factorial domain \cite{car}, and we compute its complete set of Betti elements.  Note that if an element $x\in D$ is divisible by a prime element, then its factorization graph is connected, and hence $x$ is not a Betti element.  

So let $x$ be a nonzero non-unit of $D$ that is not divisible by a prime and $(x)$ the principal ideal it generates in $D$.  As $D$ is a Dedekind domain, the ideals of $D$ factor uniquely as a product of prime ideals.  Hence, we have in $D$ that $(x)=P_1\cdots P_k$ where the $P_i$'s are non-principal prime ideals. Since the class number is 2, a nonprime irreducible $\delta$ of $D$ yields an ideal factorization into prime ideals of the form $(\delta)=P_1P_2$.  Hence the $k$ above must be even.  If $x$ is a Betti element, then we argue that $k=4$.

If $k=2$, then $x$ is irreducible and not a Betti element.  Hence, $k\geq 4$.
Suppose that $k>4$ and $x$ has more than one factorization, say 
\[
x=\alpha_1\cdots \alpha_t = \beta_1\cdots \beta_t
\]
where $t\geq 3$ (both factorizations have the same length because $D$ is half-factorial).  
If we set $z=\alpha_1\cdots \alpha_t$ and $z^\prime=\beta_1\cdots \beta_t$, then either $z$ and $z^\prime$ share an atom, or we construct a factorization $z^{\prime\prime}$ that shares atoms with $z$ and $z^\prime$. In either case, this will place them in the same connected component of $\nabla_x$.

Suppose $z$ and $z^\prime$ do not share an atom. 
Decompose the ideal $(x)$ into prime ideals as follows,
\[
(x)=(\alpha_1)\cdots (\alpha_t)= (P_{i_1}P_{i_2})\cdots (P_{i_{k-1}}P_{i_k})
\] and 
\[
(x)=(\beta_1)\cdots (\beta_t)= (P_{j_1}P_{j_2})\cdots (P_{j_{k-1}}P_{j_k}).
\] 
Since $k\ge 6$, there must be some $s,t\notin \{i_1, i_2\}$ with $(P_{s}P_{t})=(\beta_j)$ for some $j$.  So $(P_{i_1}P_{i_2})(P_{s}P_{t})$ can be extended to another factorization of $(x)$ and hence to another factorization $z^{\prime\prime}$ of $x$. 
But now $z$ and $z^{\prime\prime}$ share the atom $\alpha_1$ (where $(\alpha_1)=(P_{i_1}P_{i_2})$), while $z^{\prime\prime}$ and $z^\prime$ share the atom $\gamma$ (where $(\gamma)=(P_{s}P_{t})$).  Thus, $\nabla_x$ is connected and $x$ is not a Betti element.

Hence $k\leq 4$ and is even.  As previously noted, $k>2$, hence $k=4$.  If $Q_1, Q_2, Q_3$, and $Q_4$ are now distinct nonprincipal prime ideals of $D$, we have several possibilities for $x$.
\begin{enumerate}
\item $(x)= Q_1^4$.  Here $x$ is the square of strong atom and has only one factorization.  Hence $x$ is not a Betti element.
\item $(x)=Q_1^3Q_2$.  Here $x = \gamma_1\gamma_2$, where $(\gamma_1)=Q_1^2$ and $(\gamma_2)=Q_1Q_2$, is the only irreducible factorization of $x$, and hence $x$ is not a Betti element.
\item $(x)=Q_1^2Q_2^2$. In this case, there are two factorizations,
\[
x=\delta_1\delta_2= \eta^2
\]
where $(\delta_1)=Q_1^2$, $(\delta_2)=Q_2^2$, and $(\eta)=Q_1Q_2$.  As these factorizations are isolated, $x$ is a Betti element.
\item $(x)=Q_1^2Q_2Q_3$.  If $Q_2Q_3=(\theta)$ and $Q_1Q_3=(\kappa)$,
then the only two irreducible factorizations of $x$ are
\[
x=\gamma_1\theta = \eta\kappa.
\]
As these are isolated, $x$ again is a Betti element.
\item $(x)=Q_1Q_2Q_3Q_4$.  Here we have three irreducible factorizations of $x$.  They are
\[
x= \eta\lambda = \kappa\mu = \sigma\theta,
\]
where $(\lambda)=Q_3Q_4$, $(\mu)=Q_2Q_4$, and $(\sigma)=Q_1Q_4$.  As these factorizations are all isolated, $x$ is again a Betti element.
\end{enumerate}

Options (3), (4), and (5) above yield the complete set of Betti elements of $D$.  Hence,
\begin{multline*}
\betti(D) = \{x :  (x) = Q_1^2Q_2^2\mbox{ where }Q_1, Q_2\mbox{ are distinct nonprincipal prime ideals}\}  \\
\cup \{x :  (x) = Q_1^2Q_2Q_3\mbox{ where }Q_1, Q_2, \mbox{ and }Q_3\mbox{ are distinct nonprincipal prime ideals}\} \\
 \cup  \{x :  (x) = Q_1Q_2Q_3Q_4\mbox{ where }Q_1, Q_2, Q_3 \mbox{ and }Q_4\mbox{ are distinct nonprincipal prime ideals}\}.
\end{multline*}
An argument very similar to this can be used to compute the complete set of Betti elements in any Krull domain with divisor class group $\mathbb{Z}_2$.  It also illustrates the difficulty in such a calculation should the divisor class group have large order.
\end{example}

We close this section with an example that demonstrates the necessity of assuming that our monoids satisfy the ACCP condition.

\begin{example}\label{e:nonaccpi}
Without the ACCP assumption, the relations occurring at Betti elements need not form a presentation.  
Indeed, consider the monoid $S = \langle x_0, x_1, \ldots, y_0, y_1, \ldots, A, B, C \rangle$, with relations
\[
\text{(a)} \  (2x_i, A + 2x_{i+1}),
\qquad
\text{(b)} \  (2y_i, A + 2y_{i+1}),
\qquad \text{and} \qquad
\text{(c)} \  (B + 2x_i, C + 2y_i)
\]
for $i \ge 0$.  Although the given relations occur at the elements
\[
b_i = 2 x_i = A + 2x_{i+1} = 2A + 2x_{i+2} = \cdots,
\qquad
b_i' = 2y_i = A + 2y_{i+1} = 2A + 2y_{i+2} = \cdots,
\]
and
\begin{align*}
b_i'' &= B + b_i = B + 2x_i = B + A + 2x_{i+1} = B + 2A + 2x_{i+2} = \cdots
\\
&= C + b_i' = C + 2y_i = C + A + 2y_{i+1} = C + 2A + 2y_{i+2} = \cdots,
\end{align*}
respectively, none of the $b_i''$ are Betti elements.  Indeed, the relations~(a) connect the first row of factorizations for $b_i''$ above, while the relations~(b) connect the factorizations in the second row.  However, since some factorizations in each row have $A$ in their atomic support, the factorization graph of $b_i''$ is connected, despite the fact that without the relations~(c), there is no way to connect from the first row of factorizations to the second row.  

The underlying issue is that $b_1, b_2, \ldots$ form an infinite descending chain in the divisibility order of $S$, thereby violating the ACCP.  As a result, some factorizations in the same $\mathcal{R}$-class need not be connected by a chain of relations occurring at Betti elements.  Note that the relations specified in~(a),~(b), and~(c) form a presentation for $M$, so the $b_i$ and $b_i'$ are the only Betti elements of $M$, and thus by the above argument, any presentation of $M$ contains some relation not occurring at a Betti element.  

In view of this example, we make explicit use of the ACCP in the proof of Lemma~\ref{atom-betti}(a), which is central to the rest of the paper, ensuring this phenomenon cannot occur in  monoids satisfying the ACCP.  
\end{example}


\section{Monoids with one Betti element}
\label{sec:one-betti}

We now begin to explore the case where the monoid $S$ contains exactly one Betti element. Finitely generated monoids with a single Betti element were studied in \cite{isolated}.  The next result is a technical lemma that will be used several times later in this manuscript. Compare the second statement with \cite[Lemma~3.12]{isolated}.

\begin{lemma}\label{atom-betti}
Let $S$ be a monoid. 
\begin{enumerate}
\item[(a)] \label{basic} Let $a$ be an atom of $S$. If $a$ is not a prime, then there exists a Betti element $b$ of $S$ such that $a\le_S b$.
\item[(b)] \label{two-betti} Suppose $b,c$ are both Betti elements and $b <_S c$.  Then there is some $z\in \mathsf{Z}(c)$ with $\operatorname{supp}(z)\cap \operatorname{AS}(b)=\emptyset$.  In particular, there must be some atom $a\in\mathcal{A}(S)$ with $a\le_S c$ and $a\not \le_S b$.
\end{enumerate}
\end{lemma}
\begin{proof}
    (a) Since $a$ is not prime, there exists $x,y\in S$ such that $a\le_S x+y$ and neither $a\le_S x$ nor $a\le_S y$. This means that $x+y\in a+S$, and consequently there exists a factorization $z$ of $x+y$ such that $a$ is in the support of $z$. As $x\not\in a+S$ and $y\not\in a+S$, the atom $a$ is not in the support of any of the factorizations of $x$ and $y$. Let $u$ and $v$ be factorizations of $x$ and $y$, respectively. Then, $u+v$ is a factorization of $x+y$. It follows that $(z,u+v)\in \ker(\varphi)$. Let $\sigma$ be a minimal presentation of $\ker(\varphi)$. There exists a chain of factorizations $z_1,\dots, z_t$ such that $z_1=z$, $z_t=u+v$, and $(z_i,z_{i+1})=(a_i+c_i,b_i+c_i)$ for some $(a_i,b_i)\in \mathcal{F}(\mathcal{A}(S))^2$ and $c_i\in \mathcal{F}(\mathcal{A}(S))$ such that either $(a_i,b_i)\in \sigma$ or $(b_i,a_i)\in \sigma$ (in particular $\varphi(a_i)=\varphi(b_i)\in \operatorname{Betti}(S)$, with $\varphi$ as in \eqref{eq:fact-hom}; see \cite[Theorem~1]{acpi}, which requires $S$ to satisfy the ACCP). As $a\in \operatorname{supp}(z)$ and $a\not\in \operatorname{supp}(u+v)$, there exists $i\in \{1,\dots,t\}$ such that $a\in \operatorname{supp}(z_i)$ and $a\not\in \operatorname{supp}(z_{i+1})$. This forces $a\not\in\operatorname{supp}(c_i)$, and so $a\in \operatorname{supp}(a_i)\cup \operatorname{supp}(b_i)$. In particular, $a\le_S\varphi(a_i)\in \operatorname{Betti}(S)$.

    (b)  Assume otherwise by way of contradiction.  Then, every element of $\mathsf{Z}(c)$ shares at least one atom with $\operatorname{AS}(b)$. Because $c$ is a Betti element, it must have $z_1, z_2\in \mathsf{Z}(c)$ with $z_1,z_2$ in different connected components of $\nabla_c$.  Now, let $z_1'\in \mathsf{Z}(b)$ such that $\operatorname{supp}(z_1)\cap \operatorname{supp}(z_1')\neq \emptyset$, and similarly let $z_2'\in \mathsf{Z}(b)$ such that $\operatorname{supp}(z_2)\cap \operatorname{supp}(z_2')\neq \emptyset$.  Because $b<_Sc$ there is some $d\in S$, $d\neq 0$, with $b+d=c$.  Choose any $z_0\in \mathsf{Z}(d)$.  Now we have a chain in $\mathsf{Z}(c)$ as: $(z_1) - (z_1'+z_0) - (z_2'+z_0) - (z_2)$, where each factorization shares at least one atom with the next.  Hence $z_1,z_2$ are in the same connected component of $\nabla_c$, a contradiction.
\end{proof}

Hence, by part (a) above, a monoid with a single Betti element has full atomic support, and the second implication in \eqref{mainresult} holds.  Moreover, part (a) also allows us to prove this further property regarding the atoms in the single Betti element case.  

\begin{lemma}\label{cor:atom-divides-single-betti}
Let $S$ be a monoid such that $\operatorname{Betti}(S)=\{b\}$, and suppose that $S$ contains no primes. Call the factorizations of $b$ by $\mathsf{Z}(b)=\{x_i 
: i \in I\}$.  Define $A_i=\operatorname{supp}(x_i)$. Then, $\{A_i\}_{i\in I}$ forms a partition of $\mathcal{A}(S)$.
\end{lemma}
\begin{proof}
    By Lemma~\ref{atom-betti}(a), every $a\in \mathcal{A}(S)$ verifies $a\le_S b$. Hence $\mathcal{A}(S)=\bigcup A_i$. Since there is only a Betti element $b$, this Betti element is minimal and by \cite[Proposition~3.6]{isolated} all its factorizations are isolated. In particular, the sets $A_i$, which are the support of these factorizations, are disjoint.
\end{proof}

\begin{example}\label{generic-presentation}
It is easy to show that the second implication in \eqref{mainresult} is not reversible.  Let $S=\langle 3,5,7\rangle$. Then, $\operatorname{Betti}(S)=\{10,12,14\}$, and for all $b\in \operatorname{Betti}(S)$ and all $a\in\mathcal{A}(S)$, $b-a\in S$. \end{example}

We note that the phenomenon of Example \ref{generic-presentation} arises with every monoid with a generic presentation, i.e. all of whose atoms occur in all minimal relations.  They are uniquely presented, as shown in \cite[Proposition~5.5]{omega}). 

We now consider a factorization $z=\sum_{i\in I} n_ia_i$ of an element $d$, where the $a_i$ are atoms and $n_i$ are positive integers (and thus $I$ has finitely many elements). Define the \emph{multiplicity multiset} of $z$ as $\operatorname{M}(z)=[n_i]_{i\in I}$, where 
multiplicity is kept, but the elements are unordered.  For example $\operatorname{M}(3a+3b+c)=[3,3,1]=[1,3,3]$.  

Let $S$ be a monoid with notation as in Lemma~\ref{cor:atom-divides-single-betti}.  For each $x_i\in \mathsf{Z}(d)$, define $M_i=\operatorname{M}(x_i)$, a multiset of positive integers.  Call the set of these, 
\[\operatorname{MS}(d)=\{M_i\},\] the \textit{multiplicity shadow} of $d$.  For example, if $d=a+b=3c+2d+3e$ are the only two irreducible factorizations of $d$, then $\operatorname{MS}(d)=\{[1,1],[2,3,3]\}$.

\begin{example}
Letting $S=\langle 6,10,15\rangle$, we have 
\[
\operatorname{MS}(30)=\{[5],[3],[2]\}
\qquad \text{and} \qquad 
\operatorname{MS}(55)=\{[5,1,1],[4,1],[1,3]\}.
\]
\end{example}

In the one Betti element case, we now show that this multiplicity shadow determines the monoid up to isomorphism.

\begin{theorem} \label{char1BE} Let $S,T$ be cancellative, reduced, atomic, prime-free monoids, with $\operatorname{Betti}(S)=\{d_S\}$ and $\operatorname{Betti}(T)=\{d_T\}$.  Suppose that $\operatorname{MS}(d_S)=\operatorname{MS}(d_T)$. Then, $S$ is isomorphic to $T$.
\end{theorem}
\begin{proof}
We first construct a bijection $\tau$ from $\mathcal{A}(S)$ to $\mathcal{A}(T)$.  Note that since $\operatorname{MS}(d_S)=\operatorname{MS}(d_T)$, we must have $|\operatorname{MS}(d_S)|=|\operatorname{MS}(d_T)|$, so $|\mathsf{Z}(d_S)|=|\mathsf{Z}(d_T)|$.
Now, reorder $\mathsf{Z}(d_T)$ if necessary, so that $\mathsf{Z}(d_S)=\{x_1, x_2,\ldots\}$, $\mathsf{Z}(d_T)=\{y_1,y_2,\ldots\}$, and $\operatorname{M}(x_i)=\operatorname{M}(y_i)$ for all $i$.  
Next, for any $i$, consider $x_i=\sum_{j\ge 1} n_ja_j$ and $y_i=\sum_{j \ge 1} n_j' b_j$, where the $a_j$ are atoms of $S$ and $b_j$ are atoms of $T$.  Reorder the $a_j, b_j$ if necessary so that $n_1\le n_2\le\cdots$ and $n_1' \le n_2'\le \cdots$.  Since $\operatorname{M}(x_i)=\operatorname{M}(y_i)$, after this reordering we in fact must have $n_j'=n_j$ for all $j$.  We now define a partial bijection $\tau$ from $S$ to $T$ via $\tau(a_j)=b_j$.  
We repeat this process for each $i$.  By Lemma~\ref{cor:atom-divides-single-betti}, every atom of $S$ appears in just one $x_i$, and each atom of $T$ appears in just one $y_i$, so these partial bijections combine into a full bijection $\tau$ from $\mathcal{A}(S)$ to $\mathcal{A}(T)$.  

Second, we claim $\tau$ can be extended to a surjective homomorphism $f:S \to T$.  Linearly extend $\tau$ to an isomorphism $\psi\colon  \mathcal{F}(\mathcal{A}(S)) \to  \mathcal{F}(\mathcal{A}(T))$. Let $\varphi_S\colon \mathcal{F}(\mathcal{A}(S)) \to S$ and $\varphi_T\colon  \mathcal{F}(\mathcal{A}(T))\to T$ be the factorization homomorphisms of $S$ and $T$, respectively. According to \cite[Proposition~2.4]{grillet}, there is a unique morphism $f: S\to T$ satisfying $f \circ \varphi_S = \varphi_T\circ \psi$, provided we ensure $\ker(\varphi_S)\subseteq \ker(\varphi_T\circ \psi)$.  Indeed, by the construction of $\tau$, for each~$i$ we have $\psi(x_i)=y_i$, so $\psi$ induces a bijection $\mathsf{Z}(d_S) \to \mathsf{Z}(d_T)$.  As such, applying $\psi$ to each relation $(x_i, x_j) \in \ker(\varphi_S)$ in a minimal presentation for $S$ yields $(y_i, y_j) \in \ker(\varphi_T)$.  This ensures $\ker(\varphi_S)\subseteq \ker(\varphi_T\circ \psi)$, thereby proving the claim.  

Lastly, applying the preceding paragraph to $\tau^{-1}$ yields a homomorphism $g:T \to S$.  However, $g \circ f$ and $f \circ g$ are the identity maps on $S$ and $T$, respectively, since they restrict to the identity maps on $\mathcal A(S)$ and $\mathcal A(T)$, respectively.  This implies $f$ is an isomorphism, and the proof is complete.  
\end{proof}

A multiplicity multiset of a non-irreducible $x$ cannot be a single $1$, that is, $\operatorname{M}(x)=[1]$, since that would suggest that $x$ is in fact an irreducible.  Any other multiplicity multiset is possible, and indeed we now show that any multiplicity shadow is possible.

\begin{theorem}\label{anymultshadow} Let $MS=\{M_1, M_2, \ldots\}$ be a nonempty set containing multisets $M_i$, where each $M_i$ is a finite nonempty multiset of positive integers, not equal to $[1]$.  Then, there is a monoid $S$ with single Betti element $d$ whose multiplicity shadow satisfies $\operatorname{MS}(d)=MS$.
\end{theorem}
\begin{proof}
Let $I=\{i:M_i\in MS\}$ be the set of indices appearing in the multisets $M_i$; note that $I$ might be infinite.  For each $i$, let $n(i)$ denote the number of integers appearing in $M_i$.  We~now define $T$ as the free abelian monoid with atoms $a_{i,j}$, where $i\in I$ and $1\le j\le n(i)$. 
For each $M_i\in MS$, write $M_i=\{m_1, m_2, \ldots, m_{n(i)}\}$.  Define $d_i=\sum_{j=1}^{n(i)} m_j a_{i,j}$, and let $S$ be the quiotient of $T$ by the  congruence generated by $\{ (d_1,d_i) : i \in I\}$. Then, in $S$, each factorization $d_1, d_2, \ldots$ occurs at the same element $d$, and by construction, each factorization of $d$ is built from different atoms, so $d$ is the unique Betti element.
\end{proof}

The one Betti element case is closely related to a factorization property that has been studied widely in the recent literature (see \cite{l-f, CoyS11, GZ21}). If for each $s\in S$ and each
distinct pair of factorizations $z_1$ and $z_2$ taken from $\mathsf{Z}(s)$ we have $|z_1| \neq |z_2|$, then $S$ is
called a \textit{length-factorial monoid}.

\begin{theorem}\label{lfobe} Let $S$ be a monoid that is not factorial.  The following statements are equivalent.
\begin{enumerate}
    \item $S$ is length factorial.
    \item $S$ has one Betti element with exactly two factorizations each of different length.
\end{enumerate}
\end{theorem}

\begin{proof} 
If $S$ is length factorial and it is not factorial, then it has a single Betti element, say $d$ (\cite[Proposition~3.5]{l-f}). We know that $\ker (\pi)$ is generated by pairs of factorizations of $d$, since they are all with disjoint support (we only have one Betti element, and thus it is Betti minimal and we can apply \cite[Lemma~3.1]{isolated}). If we want $\ker(\pi)$ to be "cyclic" (\cite[Theorem~3.1]{l-f}), then we can only have two factorizations.

For the converse, if $S$ has only one Betti element with two factorizations $z$ and $z'$, then a minimal presentation for $\ker(\pi)$ is $\{(z,z')\}$, and consequently $\ker(\pi)$ is cyclic, which by \cite[Theorem~3.1]{l-f} means that $S$ is length factorial.
\end{proof}

Thus, the first implication in \eqref{mainresult} is established.  The next example show that the converse fails.

\begin{example}  Any two generated numerical monoid is length factorial, as it easily satisfies criteria (2).  On the other hand, by \cite[Example 3.4]{CJMM}, the Puiseux monoid $M := \left\langle \frac{1}{p} :  p \in \mathbb{P} \right\rangle$ has $\text{Betti}(M) = \{1\}$, but here
$2 = 3 \left(\frac{1}{3}\right) + 7 \left(\frac{1}{7}\right) = 10 \left(\frac{1}{5}\right)$ and hence $M$ is not length factorial.
\end{example}

\begin{example} \label{restrictedblock}
We shall argue shortly in Corollary~\ref{cor:blockmonoid} that $\mathbb{Z}_3$ and $\mathbb{Z}_2\oplus\mathbb{Z}_2$ are the only two abelian groups that yield a full block monoid with exactly one Betti element.  Moreover, in each of these cases the Betti element admits exactly two irreducible factorizations 
\[
\begin{array}{c}
(3,3)=(3,0)+(0,3) = (1,1) +(1,1)+(1,1) \\[5pt]
(2,2,2) = (2,0,0)+(0,2,0)+(0,0,2) = (1,1,1) + (1,1,1),
\end{array}
\]
and hence both are length factorial.
    
Here is a block monoid construction of a length factorial monoid taken from \cite[Example 7]{CS90}. We first slightly extend the notion of a block monoid.  Let $S\subseteq G$ and define
\[
\mathcal{B}(G,S) = \left\{x \in \mathcal{F}(S) : \eval(x) = 0 \right\},
\]
which we refer to as the block monoid on $G$ restricted to $S$.  It trivially follows that $\mathcal{B}(G,S)$ is a submonoid of $\mathcal{B}(G)$, and if $B$ is a Betti element of $\mathcal{B}(G,S)$, then $B$ is also a Betti element of $\mathcal{B}(G)$.

Let $n\geq 2$ be a positive integer and $~{\displaystyle G\cong \sum_{i=1}^{n}\mathbb{Z}_{n+2}}$ a finite abelian group.  Let $e_i$ be the $i$th canonical basis vector for $G$ and set 
$f=\sum_{i=1}^n -e_i$.  Let 
\[
S=\{e_1, e_2, \ldots, e_n, f\}
\]
be a subset of $G$.  We consider the restricted block monoid $\mathcal{B}(G,S)$.  Based on the calculations in \cite[Example 7]{CS90}, the irreducible elements of $\mathcal{B}(G,S)$ are $B_i=e_i^{n+2}$ (for each $1\leq i\leq n$), $C=f^{n+2}$, and
\[
D=\left(\prod_{i=1}^n e_i\right)\cdot f.
\]
It easily follows that the only Betti element of $\mathcal{B}(G,S)$ is
\[
B=\left( \prod_{i=1}^n e_i^{n+2}\right)\cdot f^{n+2}=\left(\prod_{i=1}^n B_i\right)\cdot C = D^{n+2}.
\]
Moreover, the two irreducible factorizations above are the only factorizations of $B$.  Thus $\mathcal{B}(G,S)$ has only one Betti element with exactly two irreducible factorizations with different lengths and hence is length factorial.  
\end{example}

We note that the block monoids we discuss here are Krull monoids (see \cite[Section 2.3]{GHKb}).  Length factorial Krull monoids have been characaterized in \cite{GZ21}, and this example essentially reflects all those Krull monoids with finite divisor class group. 

\begin{example}\label{morerestrictedblockmonoid}  We note that there are half-factorial monoids with one Betti element with exactly two different factorizations (this indicates that the stipulation in part (2) of Theorem~\ref{lfobe} that the Betti
element has two factorizations of different length is vital). We use the construction set up in Example~\ref{restrictedblock} to demonstrate this.
Let $t>2$ be an integer and set $n=2t$.  Consider the restricted block monoid $\mathcal{B}(\mathbb{Z}_n,S)$ with $S=\{e, te\}$ where $e$ is a generator of $\mathbb{Z}_n$. The monoid $\mathcal{B}(\mathbb{Z}_n,S)$ has three different irreducible blocks: $B=e^n$, $C=(te)^2$, and $E=e^t(te)^1$. By \cite[Corollary 3.9]{CS90}, $\mathcal{B}(\mathbb{Z}_n,S)$ is a half-factorial monoid (this can be verified using the irreducibles $B$, $C$, and $E$ and a straightforward combinatorial argument). It is also easy to verify that 
\[
\betti(\mathcal{B}(\mathbb{Z}_n,S))=\{e^n(te)^2\}
\]
where $e^n(te)^2 = (e^t(te))^2$ are the two irreducible factorizations of the Betti element. $\mathcal{B}(\mathbb{Z}_n,S)$ is a Krull monoid with divisor class group $\mathbb{Z}_n$. 

A second example can be had as follows (see \cite[Example 1.6]{CKO}).  Let
\[
M=\{(x_1, x_2, x_3, x_4)\in \mathbb{N}_0 :  x_1+x_2 = x_3+x_4\}.
\]  The irreducible elements of $M$ are $(1,0,1,0), (1, 0,0,1), (0,1,1,0),$ and $(0,1,0,1)$ and 
\[
\betti(M) = \{(1,1,1,1)\}.
\]
Here $(1,1,1,1)=(1,0,1,0)+(0,1,0,1) = (1,0,0,1)+(0,1,1,0)$ are the two irreducible factorizations of the Betti element. The set $M$ is a Krull monoid with divisor class group $\mathbb{Z}$.
\end{example}


\section{Full atomic support}
\label{sec:full-as}

Recall that an element $s$ in an atomic monoid $S$ has full atomic support if its atomic support is $\mathcal{A}(S)$, and that $S$ has full atomic support if every Betti element of $S$ has full atomic support.
We showed in Section 4 that monoids with a single Betti element have full atomic support.  Additional examples are reduced monoids with a generic presentation (see \cite{omega}). In each of these previous cases, we have that $\mathsf{c}(S)=\omega(S)=\mathsf{t}(S)$ (\cite[Theorem~5.6]{omega} and \cite[Theorem~19]{single}, respectively). In this section, we prove that this holds for all monoids with full atomic support.

\begin{example}
The numerical monoid $\langle 6,10,15\rangle$ has a single Betti element and thus has full atomic support. 
The numerical monoid $\langle 3,5,7\rangle$ has a generic presentation, and thus it has full atomic support. 
\end{example}

\begin{example}
Let $S=\langle 8, 11, 12, 13\rangle$. Then, $\operatorname{Betti}(S)=\{24, 33, 34, 38, 39\}$, and all the Betti elements of $S$ have full atomic support. The minimal presentation of $S$ is not generic. The monoids $\langle 10, 11, 15, 19\rangle$, $\langle 10, 13, 15, 17\rangle$, and $\langle  8, 19, 20, 21\rangle$ are also of full atomic support with more than one Betti element and their minimal presentations are not generic (these are the only examples with this property with Frobenius number less than 37; this search was carried out with the help of \cite{numericalsgps}).
\end{example}

\begin{example}
We now consider some examples of monoids with infinitely many Betti elements.  
Consider the Puiseux monoid $S = \langle a_i \rangle \subseteq \mathbb Q$, where 
\[
a_0 = 1, 
\quad 
a_1 = \tfrac{3}{2},
\quad \text{and} \quad 
a_i = \tfrac{1}{2}(a_{i-1} + a_{i-2})
\quad \text{for} \quad
i \ge 2.
\]
Each partial sequence $2^k a_1, \ldots, 2^k a_k$ generates a complete intersection numerical monoid, so every generator $a_i$ of $S$ is an atom, and $\betti(S) = \{2a_i : i \ge 1\}$.  In particular, $S$ has infinitely many Betti elements, all of which are incomparable to one another, although none have full atomic support.  

Next, consider the multiplicative monoid $S$ with generators $x_i$, $y_i$ for $i \in \mathbb Z$, subject to the relations $x_iy_j = x_{i+1}y_{j+1}$ for $i, j \in \mathbb Z$.  Each Betti element of $S$ is of the form $b_j = x_0y_j$ and has full atomic support.  If one instead restricts to only the relations $x_iy_j = x_{i+1}y_{j+1}$ with $1 \le j - i \le k$, then the resulting monoid has exactly $k$ Betti elements, each of which again has full atomic support.  
\end{example}

\begin{proposition}\label{prop:fr-bm}
Let $S$ be a monoid with full atomic support. Then, all its Betti elements are pairwise incomparable with respect to $\le_S$.
\end{proposition}
\begin{proof}
Let $b$ and $b'$ be Betti elements such that $b <_S b'$. By Lemma~\ref{atom-betti}(b), there exits a factorization $z$ of $b'$ whose support is disjoint with the atomic support of $b$, which is impossible as $b$ has full atomic support.
\end{proof}

\begin{example} Consider the numerical semigroup $S=\langle 5,6,9\rangle$ with Betti elements $15,18$.  These are incomparable with respect to $\le_S$, and $18$ does not have full atomic support since $5\not \le 18$.  This shows that the converse of Proposition~\ref{prop:fr-bm} does not hold.
\end{example}

Let $S$ be a monoid and let $s\in S\setminus\{0\}$. The Apéry set of $s\in S$ is defined as 
\[
\operatorname{Ap}(S,s)=\{ x\in S : x-s\not\in S\}=S\setminus(s+S).
\]
Observe that if $w\in \operatorname{Ap}(S,s)$ and $w'\in S$ is such that $w'\le_S w$, then $w'\in \operatorname{Ap}(S,s)$, that is, the set $\operatorname{Ap}(S,s)$ is divisor-closed (using additive notation).

\begin{proposition}
Let $S$ be a monoid with the ascending chain property on principal ideals. The following are equivalent.
\begin{enumerate}
\item The monoid $S$ has full atomic support.
\item For every atom $a$ of $S$, all the elements in $\operatorname{Ap}(S,a)$ have a unique factorization into products of irreducibles.
\end{enumerate} 
\end{proposition}

\begin{proof}
Suppose that there exists $a$ an atom of $S$ and $w\in \operatorname{Ap}(S,a)$ such that $w$ has more than one factorization. By \cite[Corollary~3.8]{isolated}, there exists $b\in \operatorname{Betti}(S)$ such that $b\le_S w$. But then $b\in \operatorname{Ap}(S,a)$, meaning that $b-a\not\in S$, and so $b$ does not have full atomic support.

For the converse, suppose $b$ is a Betti element that does not have full atomic support.  Then there is some atom $a$ with $b-a\notin S$, so $b\in \operatorname{Ap}(S,a)$, and, by definition, $b$ has more than one factorization.
\end{proof}

Let $S$ be a numerical monoid with multiplicity $m$.
Such $S$ having the property that all the elements in $\operatorname{Ap}(S,m)$ have unique irreducible expression, were called numerical monoids with Apéry sets of unique expression in \cite{apery-unique}, and staircase monoids in \cite{staircase}. These monoids have attracted the attention of several researchers, in part due to the properties of their minimal presentations. Other families of numerical monoids having Apéry sets of unique expression include numerical monoids with maximal embedding dimension \cite[Chapter~2]{ns} and numerical monoids with $\alpha$-rectangular Apéry sets \cite{alpha-rectangular}.

Let $S$ be a monoid, $s\in S$, and $N$ a positive integer. We now consider the notion of the distance between elements of $\mathsf{Z}(s)$.  Given $z=\sum_A z_a a,z'=\sum_A z'_a a\in \mathsf{Z}(s)$, define the \emph{distance} between $z$ and $z'$ as 
\[\operatorname{d}(z,z')=\max\{ |z-(z\wedge z')|, |z'-(z\wedge z')|\},\] 
where $z\wedge z'=\sum_{a\in A} \min\{z_a,z'_a\}a$.
An $N$-\emph{chain} of factorizations connecting $z$ and $z'$ is a sequence $z_1,\dots, z_n\in \mathsf{Z}(s)$ such that $z=z_1$, $z_n=z'$ and $\operatorname{d}(z_i,z_{i+1})\le N$ for all $i\in \{1,\dots,n-1\}$. The \emph{catenary degree} of $s$, denoted here by $\mathsf{c}(s)$, is the minimum $N\in \mathbb{N}_0\cup\{\infty\}$ such that any two factorizations of $s$ can be connected by an $N$-chain. The catenary degree of $S$ is defined as $\mathsf{c}(S)= \sup\{ \mathsf{c}(s) : s\in S\}$.

The following result shows how to compute the catenary degree of $S$ once we know the factorizations of its Betti elements.

\begin{lemma}\label{lem:c-fr}
Let $S$ be a monoid all of whose Betti elements are pairwise incomparable with respect to $\le_S$. 
Then, 
\[
\mathsf{c}(S)=\sup \{ |z| : z\in \mathsf{Z}(\operatorname{Betti}(S))\}.
\]
\end{lemma}
\begin{proof}
We already know that $\operatorname{c}(S)=\sup\{ \operatorname{c}(b) : b\in \operatorname{Betti}(S)\}$ (this is a consequence of Section~3 in \cite{c-t,ph}). In light of Proposition~\ref{prop:fr-bm}, every Betti element is minimal and by \cite[Proposition~3.6]{isolated}, all the factorizations of every Betti element are isolated, that is, each $\mathcal{R}$-class of $\mathsf{Z}(b)$ for $b\in \operatorname{Betti}(S)$ is a singleton. By \cite[Corollary~9]{ph}, $\mathsf{c}(b)=\sup\{ |z| : z\in \mathsf{Z}(b)\}$, and the result follows easily.
\end{proof}

Let $S$ be a monoid and let $s\in S$. We now consider a function that measures how far an element is from being prime.  Define the $\omega$-\emph{primality} of $s$, denoted by $\omega(s)$, as the $k\in \mathbb{N}_0\cup\{\infty\}$ such that whenever $s \le_S  a_1+\cdots+a_n$ for some $a_1,\ldots, a_n\in S$, then $s\le_S \sum_{i\in I} a_ i$ for some  $I \subseteq \{1,\ldots, n\}$ with $|I|\le k$. The $\omega$-primality of $S$ is $\omega(S)=\sup\{ \omega(a) : a \in \mathcal{A}(S)\}$.

Let $I$ be a set indices, and let $\mathbf{e}_i$ be the sequence of $\mathbb{N}_0^{(I)}$ all of whose entries are zero, except the $i$th, which is equal to one.  If we want to compute the $\omega$-primality of an atom of $S$, then by \cite[Proposition~3.3]{omega} the set $\operatorname{Minimals}_\le (\mathsf{Z}(a+S))$ is important, where $\le$ is the usual partial ordering in $\mathbb{N}_0^{(\mathcal{A}(S))}$ (which we identify with $\mathcal{F}(\mathcal{A}(S))$). We next offer a description of this set. 

\begin{lemma}\label{lem:fr-ZaS}
Let $S$ be a monoid with full atomic support, and let $a\in \mathcal{A}(S)$. Then,
\[
\operatorname{Minimals}_\le (\mathsf{Z}(a+S))= \{\mathbf{e}_a\} \cup \left\{z\in\mathsf{Z}(\operatorname{Betti}(S)) : a\not\in \operatorname{supp}(z)\right\}.
\]
\end{lemma}
\begin{proof}
Let $x\in \operatorname{Minimals}_\le(\mathsf{Z}(a+S))$, with $x\neq \mathbf{e}_a$. Then, $x_a=0$. Let $s=\varphi(x)$, with $\varphi$ as in \eqref{eq:fact-hom}. It follows that $s\in a+S$, and consequently there exists $y\in \mathsf{Z}(s)$ such that $y_a\neq 0$. Hence, $(x,y)\in \ker(\varphi)$, which implies by \cite[Theorem~1]{acpi}, that the there exists $b\in \operatorname{Betti}(S)$ and $z\in \mathsf{Z}(b)$ such that $z\le x$. In particular, $z_a=0$. As by hypothesis $b$ has full atomic support, $b\in a+S$, and so $z\in \mathsf{Z}(a+S)$. The minimality of $x$ yields $x=z\in \mathsf{Z}(\operatorname{Betti}(S))$.

For the other inclusion, let $z\in \mathsf{Z}(\operatorname{Betti}(S))$ be such $z_a=0$. If follows that $b=\varphi(z)$ is a Betti element, and consequently $\varphi(z)\in a+S$ (every Betti element of $S$ has full atomic support by hypothesis); whence $z\in \mathsf{Z}(a+S)$. If there exists $x\in \mathsf{Z}(a+S)$ with $x<z$, then $x_a=0$, and so there exists $y\in \mathsf{Z}(\varphi(x))$ such that $y_a\neq 0$. In particular, $(x,y)\in \ker(\varphi)$. Thus, by the same argument used above, there exists $b'\in \operatorname{Betti}(S)$ and $x'\in \mathsf{Z}(b')$ such that $x' \le x <x$. But then $b'=\varphi(x')\le_S \varphi(b)$ and $b\neq b'$, contradicting the fact that every Betti element of $S$ is Betti-minimal (Proposition~\ref{prop:fr-bm}).
\end{proof}

We now consider an invariant that measures how factorizations behave relative to a fixed atom.  Let $S$ be a monoid, $s\in S$, and let $a\in \mathcal{A}(S)$ be such that $a\le_S s$. There is then a factorization $z=(z_a)_{a\in \mathcal{A}(S)}$ with $z\in \mathsf{Z}(s)$ and $z_a\neq 0$. 
The \emph{tame degree} of $s$ with respect to $a$, $\mathsf{t}(s, a)$, is the least nonnegative integer $t$ such that for every $z\in\mathsf{Z}(s)$, there exists $z'=(z_a')_{a\in A}\in\mathsf{Z}(s)$ with $z'_a\neq 0$  and $\operatorname{d}(z,z')\le t$. The tame degree of $S$ with respect to $a$, denoted by $\mathsf{t}(S,a)$, is the supremum of all the tame degrees of the elements of $a+S$ with respect to $a$. 
The \emph{tame degree} of $S$, which we denote by $\mathsf{t}(S)$, is the supremum of the tame degrees of $S$ with respect to all the atoms.

\begin{theorem}\label{th:c=t=w-full-rank}
    Let $S$ be a monoid with full atomic support. Then, $\mathsf{c}(S)=\omega(S)=\mathsf{t}(S)$.
\end{theorem}
\begin{proof}
    We already know that $\mathsf{c}(S)\le \omega(S)\le \mathsf{t}(S)$ (see for instance equation (5.1) in \cite{omega}). Thus, it suffices to prove that $\mathsf{t}(S)\le \mathsf{c}(S)$. Let $s\in S$ be minimal with respect to $\le_S$ such that $\mathsf{t}(s)=\mathsf{t}(S)$, and suppose that $\mathsf{t}(s)=\operatorname{d}(z,z')$ for some $z,z'\in \mathsf{Z}(s)$. Then, $z\in \operatorname{Minimals}_\le(\mathsf{Z}(a+S))$ for some $a\in \mathcal{A}(S)$ \cite[Lemma~5.4]{omega}. By Lemma~\ref{lem:fr-ZaS}, $z\in \mathsf{z}(b)$ for some $b\in \operatorname{Betti}(S)$, and by Proposition~\ref{prop:fr-bm}, $b$ is Betti-minimal, which by \cite[Proposition~3.6]{isolated} implies that all the factorizations of $b$ are isolated. In particular, $\operatorname{d}(z,z')=\max\{|z|,|z'|\}\le \operatorname{sup}\left\{ |x| : x\in \mathsf{Z}(\operatorname{Betti}(S))\right\}= \mathsf{c}(S)$ (Lemma~\ref{lem:c-fr}).
\end{proof}

As opposed to the generic case, a dual representation is a relation with the least possible support (in the case of numerical semigroups the least possible number of atoms involved in a minimal relation is two). Relations of this form are known as circuits, (see \cite{single, u-free}).  It is known that if an affine monoid has a single Betti element, then the set of circuits is a minimal presentation of the monoid \cite[Corollary~10]{single}. 

\begin{example}
One may wonder if Theorem \ref{th:c=t=w-full-rank} holds for affine monoids having a minimal presentation formed by circuits. The answer is no and we show this with a universally free numerical monoid taken from \cite[Example~3.13]{u-free}.  Let $S=\langle 390,546,770,1155\rangle$. With the help of \cite{numericalsgps}, we can see that $\operatorname{t}(S)=\omega(S)=77$ and $\operatorname{c}(S)=55$.
\end{example}

\begin{example}
Observe that Lemma~\ref{lem:c-fr} holds for the more general class of monoids all of whose Betti elements are incomparable (and thus all its factorizations are isolated). So, one may ask if Theorem~\ref{th:c=t=w-full-rank} holds for this broader family of monoids. The answer is no, as the following example extracted from \cite{measure-w} shows. In this example, not all the Betti elements of $S$ will be of full atomic support. Take $S=\langle 17,40,56\rangle$. By using the \texttt{GAP} \cite{gap} package \texttt{numereicalsgps} \cite{numericalsgps} we see that $\operatorname{Betti}(S)=\{136,280\}$ and $280-136\not\in S$; also by \texttt{numericalsgps} we obtain that $\operatorname{c}(S)=8$, $\operatorname{t}(S)=9$, and $\omega(S)=13$.
\end{example}


\section{The special case of block monoids}
\label{sec:block}

In Example \ref{lfex} we introduced block monoids and looked at some basic calculations of their Betti elements.  We expanded on this in Section 4 by constructing examples of restricted block monoids that contained exactly one Betti element (Examples \ref{restrictedblock} and \ref{morerestrictedblockmonoid}).  In this section we characterize block monoids with one Betti element, and show that this is equivalent to the not only the block monoid having full atomic support, but also the length factorial property.  Along the way, we tie this in with several well-known factorization properties, as well as the Davenport constant of the defining finite abelian group of the block monoid.  

We open with a relatively simple, yet important lemma.

\begin{lemma}\label{full-rank-finite} 
Let $G$ be an abelian group and $S\subseteq G$.  Set $M=\mathcal{B}(G,S)$, assumed to not be factorial and with full atomic support. Then, $\mathcal{A}(M)$ and $\betti(M)$ are finite.
\end{lemma}
\begin{proof}
Let $b$ be a Betti element of $M$. Since $M$ is a submonoid of a free commutative monoid, $b$ has only finitely many divisors.  Moreover, since $b$ has full atomic support, every element of $\mathcal{A}(M)$ divides $b$.  As such, $M$ has finitely many atoms.  By \cite{redei}, $M$ then also has finitely many Betti elements.
\end{proof}

The following gives rather restrictive conditions on torsion Krull monoids all of whose Betti elements have full atomic support.  

\begin{lemma}\label{lem:blockmonoid}
Let $G$ be an abelian group and $S = \{s_1, s_2, \ldots, s_k\} \subseteq G$ a torsion subset, and assume $\mathcal{B}(G,S)$ is not factorial.  If $\mathcal{B}(G,S)$ has full atomic support, then the following hold:
\begin{enumerate}
\item 
with the exception of $a_i = \operatorname{ord}(s_i)s_i$, every atom of $\mathcal{B}(G,S)$ involves every element of $S$; and

\item 
$\langle s_1, \ldots, \hat s_j, \ldots, s_k \rangle \cong \bigoplus_{i \ne j} (\mathbb{Z}/\operatorname{ord}(s_i)\mathbb{Z})$ for each $j$.  

\end{enumerate}
\end{lemma}

\begin{proof}
Clearly each $a_i \in \mathcal A(\mathcal B(G,S))$.  Suppose $c \in \mathcal A(\mathcal B(G,S))$ is an atom distinct from the $a_i$ (note that at least one such atom must exist since $\mathcal B(G,S)$ is not factorial).  We claim every element of $S$ appears in $c$.  Indeed, suppose $s_j \in S$ appears in $c$.  Since $\mathcal{B}(G,S)$ is root-closed, $\operatorname{ord}(s_j)c$ is divisible by $a_j$ and thus has more than one factorization (see \cite[Lemma~3.14]{isolated}).  

Let $b_c \in \operatorname{Betti}(\mathcal B(G,S))$ denote the smallest multiple of $c$ with more than one factorization.  Since $b_c$ has full atomic support, it must be divisible by $a_i$ for every~$i$.  This means $b_c$ involves every element of $S$, and thus so does $c$.  

Next, up to relabeling, we can assume $j = k$, so consider the subgroup 
\[
G' = \langle s_1, \ldots, s_{k-1} \rangle = \left\{\sum_{i=1}^{k-1} m_i s_i : 0 \le m_i < \operatorname{ord}(s_i)\right\}.
\]
We claim $|G'|=\prod_{i=1}^{k-1} \operatorname{ord}(s_i)$.  Indeed, suppose $|G'|<\prod_{i=1}^{k-1}\operatorname{ord}(s_i)$.  By the pigeonhole principle, there exist two choices of coefficients that give the same element of $G'$, i.e., $\sum_{i=1}^{k-1} m_i s_i = \sum_{i=1}^{k-1} m_i' s_i$.  But then $\sum_{i=1}^{k-1}(m_i-m_i'\pmod{\operatorname{ord}(s_i)}s_i=0$, which implies there is an atom of $\mathcal{B}(G,S)$ distinct from the $a_i$ not involving $s_j$, which contradicts~(1). Hence~(2) is proved.  
\end{proof}

Lemma~\ref{lem:blockmonoid} allows us to prove the following technical result, which in some sense is a formalization of Example~\ref{lfex}.
    
\begin{theorem}\label{thm:blockmonoid}
Let $G$ be an abelian group and $S = \{s_1, s_2, \ldots, s_k\} \subseteq G$ a torsion subset, and assume $\mathcal B(G,S)$ is not factorial.  For each $i$, let $r_i$ be minimal such that $r_i s_i \in \langle s_1, \ldots, \hat s_i, \ldots, s_k \rangle$, let $t_i = r_i s_i$, and let $T = \{t_1, \ldots, t_k\}$. 
If $\mathcal{B}(G,S)$ has full atomic support, then the following hold:
\begin{enumerate}
\item 
$d := \operatorname{ord}(t_1) = \cdots = \operatorname{ord}(t_k)$ 
and 
$t_1 + \cdots + t_k = 0$;

\item
the atoms of $\mathcal B(G,S)$ are $d(t_1), \ldots, d(t_k)$, and $t_1 + \cdots + t_k$; and 

\item 
$b = d t_1 + \cdots + d t_k$ is the unique Betti element of $\mathcal B(G,S)$.  

\end{enumerate}
Moreover, $\mathcal B(G,S)$ is half-factorial if and only if $k = d$, and length-factorial otherwise.  
\end{theorem}

\begin{proof}
Since
\[
\langle s_i \rangle \cap \langle s_1, \ldots, \hat s_i, \ldots, s_k \rangle = \langle s_i^{r_k} \rangle
\]
as subgroups of $G$, in any zero-sum sequence $\sum_{i=1}^{k} m_i s_i$ we must have $r_i \mid m_i$ for each $i$.  In~particular, in any zero-sum sequence in $s_1, \ldots, s_k$, one can substitute $r_i s_i \mapsto t_i$ for each $i$ to obtain a zero-sum sequence in $t_1, \ldots, t_k$.  

We briefly address the case $k = 2$ and $t_1 = t_2$. 
In this case, one can readily check
\[
\mathcal A(\mathcal B(G,S)) = \big\{ i(r_1 s_1) + (d-i)(r_2 s_2) : 0 \le i \le d \big\},
\]
where $d = \operatorname{ord}(t_1) = \operatorname{ord}(t_2) \ge 2$.  If $d \ge 3$, then 
\[
b = 2\big( (r_1 s_1) + (d-1)(r_2 s_2) \big)
= \big( 2 (r_1 s_1) + (d-2)(r_2 s_2) \big) + \big( d(r_2 s_2) \big)
\]
is a Betti element without $d(r_1 s_1)$ in its atomic support.  This implies $d = 2$, so~(1)-(3) hold.  

Proceeding now with all other cases, we verify two additional items:
\begin{itemize}
\item 
each $t_i$ is nonzero (indeed, if some $t_j = 0$, applying Lemma~\ref{lem:blockmonoid}(2) to $S$ and $j$ implies
\[
\langle s_1, \ldots, s_k \rangle 
\cong \langle s_j \rangle \oplus \langle s_1, \ldots, \widehat s_j, \ldots, s_k \rangle
\cong (\mathbb{Z}/\operatorname{ord}(s_1)\mathbb{Z})\oplus \cdots \oplus (\mathbb{Z}/\operatorname{ord}(s_k)\mathbb{Z})
\]
which is impossible since $\mathcal B(G,S)$ is not factorial); and

\item 
the $t_i$ are distinct (indeed, if $t_i = t_j$ for some $i \ne j$, then the case $k = 2$ has already been handled, and for the case $k \ge 3$, applying Lemma~\ref{lem:blockmonoid}(2) to any $(k-1)$-element subset of $\{s_1, \ldots, s_k\}$ containing both $s_i$ and $s_j$ would force $t_i = t_j = 0$, which contradicts the preceding item).  

\end{itemize}
As such, the mapping $t_i \mapsto r_i s_i$ induces a natural isomorphism $\mathcal B(G;T) \cong \mathcal B(G,S)$.  In what follows, we will prove~(2) and~(3) by proving 
\[
\mathcal A(\mathcal B(G,T)) = \{dt_1, \ldots, dt_k, t_1 + \cdots + t_k\}
\quad \text{and} \quad
\operatorname{Betti}(\mathcal B(G,T)) = \{dt_1 + \cdots + dt_k\}.
\]
Note $\mathcal B(G,T)$ has full atomic support since $\mathcal B(G,S)$ does, and each $t_i \in \langle t_1, \ldots, \hat t_i, \ldots, t_k \rangle$.  

Now, letting $a_i = \operatorname{ord}(t_i)t_i \in \mathcal A(\mathcal B(G,T))$ for each $i$, applying Lemma~\ref{lem:blockmonoid} to $\mathcal B(G,T)$ implies any atom $c \in \mathcal A(\mathcal B(G,T))$ distinct from the $a_i$ involves every element of $T$, and the subgroup 
\[
G' = \langle t_1, \ldots, t_{k-1} \rangle \cong (\mathbb{Z}/\operatorname{ord}(t_1)\mathbb{Z})\oplus \cdots \oplus (\mathbb{Z}/\operatorname{ord}(t_{k-1})\mathbb{Z}).
\]
By construction $t_k \in G'$, so $t_k = \sum_{i=1}^{k-1} m_i t_i$ with $0 \le m_i < \operatorname{ord}(t_i)$ for each $i$.  It must be that for each $i$, $\gcd(m_i, \operatorname{ord}(t_i)) = 1$, since if this were not the case, then $t_i \notin \langle t_1, \ldots, \hat t_i, \ldots, t_k \rangle$, which is impossible.  As such, $\operatorname{ord}(t_k) = \operatorname{lcm}(\operatorname{ord}(t_1), \ldots, \operatorname{ord}(t_{k-1}))$.  However, applying Lemma~\ref{lem:blockmonoid} to any $(k-1)$-element subset of $\{t_1, \ldots, t_k\}$ yields 
\[
\operatorname{ord}(t_i) = \operatorname{lcm}(\operatorname{ord}(t_1), \ldots, \widehat{\operatorname{ord}(t_i)}, \ldots, \operatorname{ord}(t_k))
\]
for each $i$, which is only possible if $\operatorname{ord}(t_1) = \cdots = \operatorname{ord}(t_k)$.  This proves the first part of~(1).  

Let $d = \operatorname{ord}(t_1)$.  Since $t_k \in G'$, there is a unique expression of the form $\sum_{i=1}^{k-1} m_i t_i + t_k = 0$ with $0 \le m_i < d$ for each $i$.  We claim $m_i = 1$ for each $i$.  Indeed, suppose some $m_j > 1$, and let $c = \sum_{i=1}^{k-1} m_i t_i + t_k$, which is necessarily an atom due to the singular copy of $t_k$ and the  uniqueness of $m_1, \ldots, m_{k-1}$.  Since $\mathcal B(G,T)$ is root-closed, the element $(d-1)c \in \mathcal B(G,T)$ is divisible by $a_j$ and thus has a factorization with $a_j$ in its support.  As such, $(d-1)c$ is divisible by a (full atomic support) Betti element, which is impossible since $a_k$ does not divide $(d-1)c$.  This proves the claim, and thus~(1).  

Only three claims remain to be verified.  For each $j = 2, \ldots, k-1$, we see $j(t_1 + \cdots + t_k)$ is the unique element of $\mathcal B(G,T)$ of the form $\sum_{i=1}^{k-1} m_i t_i + jt_k$ with $0 \le m_i < d$ for each $i$, and thus not an atom, so~(2) is proven.  The uniqueness of the Betti element in~(3) follows from the observation that $\mathcal B(G,T)$ can be realized as a full-dimensional affine semigroup in $\mathbb Z^k$ with $k+1$ atoms.  Lastly, the half- and length-factoriality claims follow from Theorem~\ref{lfobe} and the fact that
\[
d(t_1 + \cdots + t_k) = a_1 + \cdots + a_k
\]
are the only factorizations of the only Betti element of $\mathcal B(G,T)$.  
\end{proof}

We show in the next result that for full block monoids the notion of length factorial and one Betti element coincide.

\begin{corollary}\label{cor:blockmonoid}
Let $G$ be an abelian group.  The following statements are equivalent:
\begin{enumerate}
\item $\mathcal{B}(G)$ is a non-factorial block monoid in which all Betti elements have full atomic support;
\item $\mathcal{B}(G)$ is a block monoid with exactly one Betti element;
\item  $G$ is isomorphic to $\mathbb{Z}_2\oplus \mathbb{Z}_2$ or $\mathbb{Z}_3$;
\item the Davenport constant of $G$ is equal to 3;
\item $1<\rho(\mathcal{B}(G))\leq 3/2$; and
\item $\mathcal{B}(G)$ is length factorial but not factorial.
\end{enumerate}
\end{corollary}

\begin{proof}
The equivalence of~(1) and~(2) follows directly from Theorem~\ref{thm:blockmonoid}.  

(1)$\Rightarrow$(3). First, we note that $\mathcal{A}(\mathcal{B}(G))$ is finite if and only if $G$ is finite \cite[Theorem 3.4.2 and Proposition 5.1.3]{GHKb}.  
Then, applying Theorem~\ref{thm:blockmonoid} to $\mathcal B(G, S)$ with $S = G \setminus \{0\} = \{s_1, \ldots, s_k\}$, we must have $r_i = 1$ for all $i$.  Now,  if $k = 2$, then $s_1 = -s_2$ and thus $G \cong \mathbb Z_3$.  If, on the other hand, $k \ge 3$, then Lemma~\ref{lem:blockmonoid}(2) implies $s_1 + s_2 \notin \{s_1, \ldots, s_{k-1}\}$ and thus $s_1 + s_2 = s_k$.  Theorem~\ref{thm:blockmonoid}  thenforces $k = 3$ and $s_3 = s_1 + s_2$, and so $G \cong \mathbb Z_2 \oplus \mathbb Z_2$.

(3)$\Rightarrow$(4). By \cite[Theorem 1]{GS92}, $\mathsf{D}(\mathbb{Z}_2\oplus \mathbb{Z}_2)=\mathsf{D}(\mathbb{Z}_3)=3$.

(4)$\Rightarrow$(5). By \cite[Proposition 5.5]{BC11}, $\mathcal{B}(G) = \frac{3}{2}$.

(5)$\Rightarrow$(6). Again by \cite[Proposition 5.5]{BC11}, $\rho(\mathcal{B}(G) = \frac{\mathsf{D}(G)}{2}$ where $\mathsf{D}(G)$ is an integer greater than or equal to 2.  By (4), our only option is $\rho(\mathcal{B}(G) = \frac{3}{2}$.  Thus, $\mathsf{D}(G) =3$ and again appealing to \cite[Theorem 1]{GS92}, we obtain that $G=\mathbb{Z}_3$ or $G=\mathbb{Z}_2\oplus\mathbb{Z}_2$. By Example~\ref{lfex}, Both $\mathcal{B}(\mathbb{Z}_3)$ and $\mathcal{B}(\mathbb{Z}_2\oplus\mathbb{Z}_2)$ are length factorial.      

(6)$\Rightarrow$(2). This follows from Proposition~\ref{lfobe}.
\end{proof}

\section*{Acknowledgements}

The second author was partially supported by the Proyecto de Excelencia de la Junta de Andaluc\'ia grant ProyExcel\_00868, the Junta de Andaluc\'ia grant FQM--343, and the Spanish Ministry of Science and Innovation (MICINN) grant ``Severo Ochoa and María de Maeztu Programme for Centres and Unities of Excellence'' (CEX2020-001105-M).
The~second and third authors received financial support from PID2022-138906NB-C21, funded by the Agencia Estatal de Investigaci\'on (AEI) grant MICIU/AEI/10.13039/501100011033, as well as from the European Regional Development Fund (ERDF) program ``A way of making Europe.'' 
The~visit of the third author to the University of Granada was also supported by the University of Granada research fund program ``Estancias de Investigadores de otros Centros Nacionales y Extranjeros en Departamentos e Institutos o Centros de Investigaci\'on de la Universidad de Granada.''


\begin{thebibliography}{9}

\bibitem{ABLST} K. Ajran, J. Bringas, B. Li, E. Singer, and M. Tirador,  Factorization in additive monoids of evaluation polynomial semirings, \textit{Comm. Algebra} \textbf{51} (2023), 4347--4362.

\bibitem{BC11} P. Baginski and S. T. Chapman, Factorizations of algebraic integers, block monoids, and additive number theory, \textit{Amer. Math. Monthly} \textbf{118}, (2011), 901--920.

\bibitem{staircase} C. Brower, J. McDonough, C. O'Neill, Numerical semigroups, polyhedra, and posets IV: walking the faces of the Kunz cone,  arXiv:2401.06025. 

\bibitem{acpi} M. Bullejos, P. A. García-Sánchez, Minimal presentations for monoids with the ascending chain condition on principal ideals, \textit{Semigroup Forum} \textbf{85} (2012), 185--190.

\bibitem{omega} V. Blanco, P. A. García-Sánchez, A. Geroldinger, Semigroup-theoretical characterizations of arithmetical invariants with applications to numerical monoids and Krull monoids, \textit{Illinois J. Math.} \textbf{55} (2011), 1385--1414.

\bibitem{car}  L. Carlitz, A characterization of algebraic number fields with class number two,
\textit{Proc. Amer. Math. Soc.} \textbf{11} (1960), 391–-392.

\bibitem{delta-bf} S. T. Chapman, P. A. García-Sánchez, D. Llena, A. Malyshev, D. Steinberg, On the Delta set and the Betti elements of a BF-monoid, \textit{Arab. J. Math.} \textbf{1} (2012), 53–61.

\bibitem{c-t} S. T. Chapman, P. A. García-Sánchez, D. Llena, V. Ponomarenko, and J. C. Rosales, The catenary and tame degree in finitely generated commutative cancellative monoids, \textit{Manuscripta Math.} \textbf{120} (2006), 253--264.

\bibitem{measure-w} S. T. Chapman, P. A. García-Sánchez, Z. Tripp, C. Viola, Measuring primality in numerical semigroups with embedding dimension three, \textit{J. Algebra Appl.} \textbf{15} (2016), 1650007 (16 pages).


\bibitem{l-f} {S. T. Chapman, Jim Coykendall, Felix Gotti, William W. Smith, Length-factoriality in commutative monoids and integral domains, \textit{J. Algebra} \textbf{578} (2021), 186--212}

\bibitem{CJMM} S. T. Chapman, J. Jang, J. Mao, S. Mao, On the set of Betti elements of a Puiseux monoid, to appear in \textit{Bull. Aus. Math. Soc.} doi:10.1017/S0004972724000352

\bibitem{CKO} S. T. Chapman, U. Krause, and E. Oeljeklaus, On Diophantine monoids and their class groups, \textit{Pacific J. Math.} \textbf{207} (2002), 125--147. 

\bibitem{CS90} S. T. Chapman and W. W. Smith, Factorization in Dedekind domains with finite class group, \textit{Israel J. Math.} \textbf{71} (1990), 65--95.


\bibitem{CG22} J. Correa-Morris and F. Gotti: \emph{On the additive structure of algebraic valuations of polynomial semirings}, \textit{J. Pure Appl. Algebra} \textbf{226} (2022) 107104.

\bibitem{CoyS11} J. Coykendall and W.W. Smith, On unique factorization domains, \textit{J. Algebra} \textbf{332} (2011), 62--70.

\bibitem{alpha-rectangular} M. D’Anna, V. Micale, A.  Sammartano, A, Classes of complete intersection numerical semigroups. \textit{Semigroup Forum} \textbf{88} (2014), 453--467.

\bibitem{numericalsgps} M. Delgado, P. A. García-Sánchez, J. Morais, NumericalSgps, A package  for  numerical  semigroups,  Version 1.3.1 dev (2023), Refereed GAP package, \url{https://gap-packages.github.io/numericalsgps}.

\bibitem{gap} The GAP~Group, GAP -- Groups, Algorithms, and Programming, Version 4.12.2,  (2022), \url{https://www.gap-system.org}.

\bibitem{u-free} I. García-Marco, P. A. García-Sánchez, I. Ojeda, Ch. Tatakis, Universally free numerical semigroups, \textit{J. Pure Appl. Algebra} \textbf{228} (5) (2024), Article No. 107551 (24 pages).

\bibitem{isolated}P. A. Garc\'{\i}a-S\'anchez, A. Herrera-Poyatos,  Isolated factorizations and their applications in simplicial affine semigroups, \textit{J. Algebra Appl.} \textbf{19} (2020), 2050082 (42 pages).


\bibitem{GSMC} P. A. García-Sánchez, H. Martín-Cruz H., Numerical semigroups with embedding dimension three and minimal catenary degree, \textit{Integers} \textbf{20}(2020), \#A81.

\bibitem{half-factorial} P. A. García-Sánchez, I. Ojeda, A. Sánchez-R.-Navarro, Factorization invariants in half-factorial affine semigroups, \textit{Internat. J. Algebra Comput.} \textbf{23} (2013), 111--122.

\bibitem{single} P. A. García-Sánchez, I. Ojeda, J. C. Rosales, Affine semigroups having a unique Betti element, \textit{J. Algebra Appl.} \textbf{12} (2013), 1250177 (11 pages).

\bibitem{GHKb} A. Geroldinger and F. Halter-Koch, {\it Non-Unique Factorizations: Algebraic, Combinatorial, and Analytic Theory}, Chapman and Hall/CRC, Boca Raton, Florida, 2006.

\bibitem{GS92} A. Geroldinger and R. Schneider, On Davenport's constant. \textit{J. Comb. Theory, Series A}, \textbf{61} (1992), 147--152.

\bibitem{GZ21} A. Geroldinger and Q. Zhong, A characterization of length-factorial Krull monoids, \textit{New York J. Math.} \textbf{27} (2021), 1347--1374.

\bibitem{GG18} F. Gotti and M. Gotti: \emph{Atomicity and boundedness of monotone Puiseux monoids}, \textit{Semigroup Forum} \textbf{96} (2018), 536--552.

\bibitem{grillet} P. A. Grillet, Commutative semigroups, Advances in Mathematics (Dordrecht) \textbf{2} (2001). Kluwer Academic Publishers, xiv, 436 p. 

\bibitem{Her} J. Herzog, Generators and relations of abelian semigroups and semigroup rings, \textit{Manuscripta Math.} \textbf{3}(2)(1970), 175--193.


\bibitem{ph} A. Philipp, A characterization of arithmetical invariants by the monoid of relations, \textit{Semigroup Forum} \textbf{81} (2010), 424--434.

\bibitem{apery-unique} J. C. Rosales, Numerical semigroups with Apéry sets of unique expression, \textit{J. Algebra} \textbf{226} (2000), 479--487.


\bibitem{redei} L. Rédei, The theory of finitely generated commutative semigroups, Pergamon Press, 1965.

\bibitem{ns} J. C. Rosales and  P. A. García-Sánchez, Numerical semigroups, Developments in Mathematics, \textbf{20}, Springer, New York, 2009.

\bibitem{elasticity-atomic} J. C. Rosales, P. A. García-Sánchez, J. I. García-García, Atomic commutative monoids and their elasticity, \textit{Semigroup Forum} \textbf{68} (2004), 64--86.

\end{thebibliography}
\end{document}